\documentclass[12pt]{amsart}

\usepackage[all]{xy}
\usepackage{amssymb}
\usepackage{amsmath}
\usepackage{stmaryrd}
\usepackage{bbm}
\usepackage{txfonts}
\usepackage{tikz-cd}
\usepackage{url}
\usepackage{etoolbox}

\makeatletter
\patchcmd{\@addmarginpar}{\ifodd\c@page}{\ifodd\c@page\@tempcnta\m@ne}{}{}
\makeatother
\reversemarginpar

\usepackage[OT2,T1]{fontenc}
\DeclareSymbolFont{cyrletters}{OT2}{wncyr}{m}{n}
\DeclareMathSymbol{\Sha}{\mathalpha}{cyrletters}{"58}

\newcommand{\segment}[2]{\subsection{#2}\label{#1}}
\newcommand{\ssegment}[2]{\subsubsection{#2}\label{#1}}

\theoremstyle{definition}
\newtheorem*{prop*}{Proposition}

\theoremstyle{definition}
\newtheorem*{thm*}{Theorem}

\theoremstyle{definition}
\newtheorem*{cor*}{Corollary}

\theoremstyle{definition} 
\newtheorem{ssprop}[subsubsection]{Proposition}

\theoremstyle{definition} 
\newtheorem{sprop}[subsection]{Proposition}

\theoremstyle{definition} 
\newtheorem{sconj}[subsection]{Conjecture}

\theoremstyle{definition} 
\newtheorem{sthm}[subsection]{Theorem}

\theoremstyle{definition} 
\newtheorem{ssthm}[subsubsection]{Theorem}

\theoremstyle{definition} 
\newtheorem{sslm}[subsubsection]{Lemma}

\theoremstyle{definition} 
\newtheorem{slm}[subsection]{Lemma}

\theoremstyle{definition} 
\newtheorem{scor}[subsection]{Corollary}

\theoremstyle{definition} 

\newcommand{\ssCorollary}[2]{\begin{scor} \label{#1} 
{#2} \end{sscor}}

\theoremstyle{definition} 
\newtheorem{ssconj}[subsubsection]{Conjecture}

\theoremstyle{definition}
\newtheorem*{conj*}{Conjecture}

\theoremstyle{definition} 
\newtheorem{sscond}[subsubsection]{Condition}

\theoremstyle{definition} 
\newtheorem{ssrmk}[subsubsection]{Remark}

\let\oldmarginpar\marginpar
\renewcommand\marginpar[1]{\-\oldmarginpar[\raggedleft\footnotesize #1]%
{\raggedright\footnotesize #1}}

\newcommand{\from}{\leftarrow}
\newcommand{\xto}{\xrightarrow}

\renewcommand{\Im}{\operatorname{Im}}

\newcommand{\Spec}{\operatorname{Spec}}

\newcommand{\Hom}{\operatorname{Hom}}
\newcommand{\Ext}{\operatorname{Ext}}

\newcommand{\Aut}{\mathrm{Aut}\,}

\newcommand{\m}[1]{\mathrm{#1}}

\newcommand{\bb}[1]{\mathbb{#1}}
\newcommand{\cl}[1]{\mathcal{#1}}

\newcommand{\la}{\lambda}
\newcommand{\ka}{\kappa}

\newcommand{\Si}{\Sigma}
\newcommand{\ze}{\zeta}

\newcommand{\al}{\alpha}

\newcommand{\Om}{\Omega}
\newcommand{\ep}{\epsilon}
\newcommand{\de}{\delta}

\newcommand{\Cc}{\mathcal{C}}

\newcommand{\ZZ}{\bb{Z}}

\newcommand{\DD}{\bb{D}}
\newcommand{\NN}{\bb{N}}
\newcommand{\QQ}{\bb{Q}}

\newcommand{\PP}{\bb{P}}
\newcommand{\KK}{\bb{K}}

\newcommand{\Bb}{\mathcal{B}}

\newcommand{\Ff}{\mathcal{F}}

\newcommand{\Hh}{\mathcal{H}}
\renewcommand{\AA}{\bb{A}}
\newcommand{\Yy}{\mathcal{Y}}

\newcommand{\Oo}{\mathcal{O}}
\newcommand{\Kk}{\mathcal{K}}

\newcommand{\Aa}{\mathcal{A}}
\newcommand{\Ee}{\mathcal{E}}
\newcommand{\Pp}{\mathcal{P}}
\newcommand{\Dd}{\mathcal{D}}

\newcommand{\Ss}{\mathcal{S}}

\newcommand{\inv}{^{-1}}

\newcommand{\areq}{\ar@{=}}
\newcommand{\suphook}{\ar@{^(->}}
\newcommand{\subhook}{\ar@{_(->}}

\newcommand{\smses}[6]
{
\[
\xymatrix{
1 \ar[r] &
#1 \ar[r]_-{#2} &
#3 \ar[r]_-{#4} &
#5 \ar[r] \ar@/_1.5pc/[l]_-{#6} &
1
}
\]
}

\newcommand{\thrpl}{\PP^1 \setminus \{0,1,\infty\}}

\newcommand{\et}{{\textrm {\'et}}}

\newcommand{\un}{\m{un}}

\newcommand{\DA}{\operatorname{DA}}

\newcommand{\Ho}{\operatorname{Ho}}

\newcommand{\Mdga}{\m{Mdga}}

\newcommand{\Sm}{\operatorname{Sm}}
\newcommand{\op}{^\m{op}}

\newcommand{\DMdga}{\operatorname{DMdga}}

\newcommand{\one}{\mathbbm{1}}

\newcommand{\DDMdga}{\operatorname{\DD Mdga}}

\newcommand{\colim}{\operatorname{colim}}

\newcommand{\Set}{\operatorname{Set}}

\newcommand{\Ch}{\operatorname{Ch}}
\newcommand{\DDA}{\operatorname{\DD A}}

\newcommand{\CAlg}{\operatorname{CAlg}}
\newcommand{\ChM}{\operatorname{ChM}}

\newcommand{\SmPr}{\operatorname{SmPr}}

\newcommand{\CCoalg}{\operatorname{CCoalg}}

\newcommand{\MChC}{\operatorname{MChC}}

\newcommand{\Id}{\operatorname{Id}}

\newcommand{\kbar}{{\overline k}}

\renewcommand{\cl}{\operatorname{cl}}

\newcommand{\Xbar}{\overline X}

\newcommand{\gr}{\mathrm{gr}}

\newcommand{\Fin}{\mathrm{Fin}}

\newcommand{\Finst}{\Fin_*}

\newcommand{\Fun}{\operatorname{Fun}}

\newcommand{\Lan}{\operatorname{Lan}}

\newcommand{\Tot}{\operatorname{Tot}}

\newcommand{\cosk}{\operatorname{cosk}}

\renewcommand{\Pr}{\operatorname{Pr}}

\title{Morphisms of rational motivic homotopy types}

\author{Ishai Dan-Cohen and Tomer Schlank}

\thanks{I.D. was supported by an ISF grant for work ``Around Kim's conjecture: from homotopical foundations to algorithmic applications.'' T.S. was supported by an Alon Fellowship.}

\date{\today} 

\begin{document}

\maketitle

\raggedbottom
\SelectTips{cm}{11}

\begin{abstract}

We investigate several interrelated foundational questions pertaining to the study of motivic dga's of Dan-Cohen--Schlank \cite{RMPS} and Iwanari \cite{Iwanari}.
 In particular, we note that morphisms of motivic dga's can reasonably be thought of as a nonabelian analog of motivic cohomology. Just as abelian motivic cohomology is a homotopy group of a spectrum coming from K-theory, the space of morphisms of motivic dga's is a certain limit of such spectra; we give an explicit formula for this limit --- a possible first step towards explicit computations or dimension bounds. We also consider commutative comonoids in Chow motives, which we call ``motivic Chow coalgebras''. We discuss the relationship between motivic Chow coalgebras and motivic dga's of smooth proper schemes.

As a small first application of our results, we show that among schemes which are finite \'etale over a number field, morphisms of associated motivic dga's are no different than morphisms of schemes. This may be regarded as a small consequence of a plausible generalization of Kim's relative unipotent section conjecture, hence as an ounce of evidence for the latter.


\end{abstract}

\setcounter{tocdepth}{1}
\tableofcontents


\section{Introduction}

\segment{mad2}{Overview}

Let $X$ be a smooth scheme over a base $Z$. When $k$ is a number field and $Z = \Spec k$ is the associated ``number scheme'', a section of the projection
\[
X \to Z
\]
is often referred to as a \emph{rational point}. When instead $Z$ is an open subscheme of $\Spec \Oo_k$ (an ``open integer scheme''), such a section is often called an \emph{integral point}. Regardless of the particular setting, our theory of \textit{motivic dga's} gives rise to a weaker notion of point which we call \emph{rational motivic point}. This theory is based on a certain functor
\[
C^* = C^*_{\DMdga}: \Sm(Z)\op \to \DMdga(Z,\QQ)
\]
from the category of smooth schemes over $Z$ to a category of \emph{motivic dga's}, which may be thought of as a motivic avatar of rational homotopy theory. To construct it, we consider a \emph{model} $M(Z,\QQ)$ for the tensor triangulated category $\DA(Z,\QQ)$ of motivic complexes over $Z$ with $\QQ$ coefficients. The category $\Mdga(Z,\QQ)$ of commutative monoids in $M(Z,\QQ)$ inherits a model structure, and we let $\DMdga(Z,\QQ)$ be its homotopy category. The ensuing theory of motivic dga's first grew out of conversations with Marc Levine and Gabriela Guzman and developed in \textit{Rational motivic path spaces}... \cite{RMPS}. Closely related theories have been simultaneously under development in works by Isamu Iwanari \cite{Iwanari} and Gabriela Guzman \cite{Guzman}.

In terms of the category $\DMdga(Z,\QQ)$, and the functor $C^*$, we define a \emph{rational motivic point of $X$} to be an augmentation
\[
C^*X \to \one
\]
of $C^*X$ in $\DMdga(Z,\QQ)$.\footnote{
While the adjective \emph{rational} in the term \emph{rational point} refers to the base, its meaning in the term \emph{rational motivic point} is different --- it refers to the coefficients, and to the effect they have on \emph{rationalizing} spaces.
}
 Being a functor, $C^*$ gives rise to a map from \emph{points} (rational, integral,...) to \emph{rational motivic points}.

Motivation for this terminology comes from the theory of the \emph{unipotent fundamental group} \cite{Deligne89, DelGon} and from Minhyong Kim's nonabelian refinement of Chabauty's method \cite{kimi, kimii}. Kim's theory pertains to the case that $X$ is a hyperbolic curve, and shows (in a growing list of cases) that the prounipotent groupoid of homotopy classes of motivic paths connecting the integral points to a fixed base-point contains enough information to single out the integral points inside the space $X(Z_p)$ of $p$-adic points. In turn, our work \textit{Rational motivic path spaces} \cite{RMPS} shows that rational motivic points contain enough information to retrieve the groupoid of homotopy classes of motivic paths. 

If a morphism
$
C^*X \to \one
$
in $\DMdga(Z,\QQ)$ contains enough information to merit the name \textit{rational motivic point}, a morphism
\[
C^*X \to C^*Y
\]
for $Y$ another smooth scheme over $Z$ may reasonably be called \emph{rational motivic morphism from $Y$ to $X$}. Our goal in this work is to study the set
\[
\Hom_{\DMdga(Z,\QQ)}(C^*X, C^*Y)
\]
of rational motivic morphisms. Our main result gives rise to a homotopy spectral sequence which converges to this pointed set; its first pages consist of certain groups constructed out of rational K-theory.

When $Z = \Spec k$ is a number scheme and $X$, $Y$ are both proper over $Z$, another notion of rational morphism with a motivic flavor presents itself. Let $\MChC(Z,\QQ)$ be the category of commutative comonoids in the category $\ChM(Z,\QQ)$ of Chow motives over $Z$ with $\QQ$ coefficients. The usual functor from the category $\SmPr(Z)$ of smooth proper schemes over $Z$ to $\ChM(Z,\QQ)$ lifts to a functor 
\[
C_* = C_*^{\MChC}: \SmPr(Z) \to \MChC(Z,\QQ).
\]
We refer to the objects of $\MChC(Z, \QQ)$ as \emph{motivic Chow coalgebras}
\footnote{
Drawing on the analogy between spectra and abelian groups, an object of a stable monoidal infinity category equipped with a unit morphism and a single binary operation (associative up to coherent higher homotopies) is often called an \emph{algebra}. Although $\ChM(Z,\QQ)$ is not itself stable, it is a full subcategory of the homotopy category of a stable monoidal infinity category. 
}
\footnote{
Actually, philosophically, a comonoid in Chow motives should only be called a motivic Chow coalgebra if it \emph{comes from geometry} in some sense. In the present work, we will only consider comonoids which come from geometry in an obvious sense. 
} and to a morphism
\[
C_*Y \to C_*X
\] 
in $\MChC(Z, \QQ)$ as a \emph{rational Chow morphism from $Y$ to $X$}. The functor $ C_*^{\MChC}$ factors through $C^*_{\DMdga}$, so that rational Chow morphisms, of interest in their own right, retain no more information than do rational motivic morphisms. Since everything in the article is \textit{rational}, we will allow ourselves to drop this adjective from our terminology, and speak of \emph{motivic morphisms} and of \emph{Chow morphisms}.

A small first application takes place in dimension 0. In the case of $X$, $Y$ finite \'etale over $Z = \Spec k$ a field of characteristic zero, a simple direct computation shows that 
\[
\Hom_Z(Y,X) = \Hom_{\MChC}(C_*Y, C_*X).
\]
(Proposition \ref{mcffe}). Combining this with our Theorem \ref{nmc}, we obtain the analogous (stronger) statement concerning motivic morphisms:

\begin{prop*}[\ref{mdffe} below]
Suppose $Z = \Spec k$ is a field and $X$, $Y$ are finite \'etale over $Z$. Then
\[
\Hom_Z(Y,X) = \Hom_{\DMdga}(C^*X, C^*Y).
\]
\end{prop*}

\noindent
Although we are not ready to state a precise conjecture at this point, we do not expect a similar result for positive dimensional schemes. 

\segment{}{Relation to Chabauty--Kim theory}
\label{mad2}

This work is the second installment in our ongoing investigation of \textit{Kim's conjecture}. We refer the reader to Kim et al. \cite{nats} for the statement of the conjecture (and to \textit{Rational motivic path spaces} \cite{RMPS} for an overview of our program) and restrict attention here to the conjecture's relevant consequences.

To fix ideas, let's focus on the case $Z \subset \Spec \ZZ$ open. If $p$ is a prime of $Z$, an analogous construction to the one outlined above gives rise to a category $\DMdga(Z_p, \QQ)$ of motivic dga's over the complete local scheme $Z_p$ of $Z$ at $p$, and the associated functor $C^*_{\DMdga}$ commutes with the global one. We thus obtain a commuting square
\[
\tag{*}
\xymatrix{
X(Z) \ar[r] \ar[d] & X(Z_p) \ar[d]^\al
\\
\Hom_{\DMdga(Z,\QQ)}
	(C^*X, \one)
	\ar[r]_-{\la} &
	\Hom_{\DMdga(Z_p,\QQ)}
	(C^*X_p, \one).
}
\]
Here $X_p$ denotes the base-change of $X$ to $Z_p$. Kim's conjecture implies that when $X$ is a hyperbolic curve, the natural injection of sets
\[
X(Z) \subset \al\inv(\Im \la)
\]
from integral points to ``motivically integral points'' is in fact bijective:

\begin{ssconj}
\label{mad2.5}
In the situation and the notation of segment \ref{mad2} with $X$ a hyperbolic curve over $Z$, we have
\[
X(Z) = \al\inv(\Im \la).
\]
\end{ssconj}

\ssegment{mad2b1}{}
In fact, in the setting of conjecture \ref{mad2.5}, there appears to be little reason to restrict attention to sections of the structure morphism $X \to Z$. We could equally consider two smooth schemes $X$ and $Y$ over $Z$ and the associated square 
\[
\tag{$*_Y$}
\xymatrix{
\Hom_Z(Y,X)
	\ar[r] \ar[d] & 
	\Hom_{Z_p}(Y_p,X_p)
	\ar[d]^\al
\\
\Hom_{\DMdga(Z,\QQ)}
	(C^*X, C^*Y)
	\ar[r]_-{\la} &
	\Hom_{\DMdga(Z_p,\QQ)}
	(C^*X_p, C^*Y_p),
}
\]
and we may hope that a bijection 
\[
\tag{$**_Y$}
\Hom_Z(Y,X) = \al\inv(\Im \la)
\]
holds \emph{for many schemes $X,Y$}. Ineed, Kim's conjecture takes part of its inspiration from Grothendieck's section conjecture. The latter, in turn, forms a small part of Grothendieck's anabelian geometry. In this setting, the functor $\pi_1$, suitably interpreted, is expected to be close to a fully faithful embedding of a large part of the category of $Z$-schemes in the category of profinite groups. To our knowledge, it is reasonable to expect a similar generalization of Kim's conjecture, which would thus apply to morphisms of anabelian schemes. 

It is worth noting, however, that while an anabelian scheme is supposed to represent the arithmetic analog of a $K(\pi,1)$, equation ($**_Y$) bears no particular relationship to $\pi_1$. So we should expect the class of schemes for which it holds to be, if anything, larger than the class of anabelian schemes.

\segment{i2}{Relation to K-theory, main theorem}

The set
\[
\Hom_{\DMdga(Z,\QQ)}
	(C^*X, C^*Y)
\]
presents itself naturally as the set of connected components $\pi_0$ of a topological space. The space in question is the space of morphisms
\[
\Hom_{\DDMdga(Z,\QQ)}
	(C^*X, C^*Y)
\]
in the infinity category
\[
\DDMdga(Z,\QQ)
\]
associated to the model category $\Mdga(Z,\QQ)$. Monadic resolution allows us to relate this space to the spaces
\[
\Hom_{\DDA(Z,\QQ)}
	(C^*X^n, C^*Y)
\]
of morphisms in the infinity category $\DDA(Z,\QQ)$ associated to the model category $M(Z,\QQ)$. In turn, by work of Morel-Voevodsky \cite{MorVod}, Ayoub \cite{AyoubSixI, AyoubSixII}, Cisinski-D\'eglise \cite{CisDeg}, these latter spaces are the infinite loop spaces associated to Adams eigenspaces of rational K-theory spectra
\[
\Kk^{(i)}(X^n \times Y),
\]
at least when $X$ is proper.
\footnote{When $X$ is not proper, one must take the cofiber of the complement inside a compactification; see proposition \ref{mosc1} for the precise statement under stringent assumptions on $X$.}
\footnote{In segment \ref{ad1} below, we \textit{define} $\Kk^{(i)} := \QQ(i)[2i]$ and discuss the relationship to K-theory. See Riou \cite{RiouThese} for a less tautological definition.}

The resulting formula is our main theorem:

\begin{ssthm}[Comparison with K-theory]
\label{elal1}
Suppose $X$, $Y$ are smooth schemes over $Z$. Assume $X$ is proper. Then we have
 \[
\Hom_{\DDMdga}(C^*X, C^*Y) =
 \lim_{k \in \Delta}
 \Omega^\infty
  \prod_{[S = (S_1 \to \cdots \to S_k)] \in \pi_0(\Fin^k_\gr)}
  \big(  \Kk^{(ds_1)} (X^{S_1}\times Y) \big)^{\Aut S}.
 \]
Here $\Fin^k_\gr$ denotes the (1-)category whose objects are sequences of morphisms of finite sets and whose morphisms are isomorphisms of diagrams, and $s_1$ denotes the cardinality of the finite set $S_1$.\footnote{We place an $\Om^\infty$ between the two limits in order to emphasize the fact that while the inner limit may be taken inside the category of spectra, the $\Delta$-shaped diagram of the outer limit contains morphisms which are \emph{not} morphisms of spectra.}
\end{ssthm}

\noindent
The full statement, in theorem \ref{nmc} below, includes a more general case in which $X$ is not proper (including for instance $\thrpl$), and gives a ``K-theoretic'' construction of the cosimplicial space appearing under the outer limit. As a corollary we obtain a spectral sequence whose first few pages may be described in terms of the K-groups
\[
K^{(dn)}_{i} (X^n \times Y) 
\]
and which abutts to the homotopy groups of the space of morphisms of motivic dga's (\S\ref{lead}).

\segment{}{Monadic resolution}
The comparison theorem (\ref{elal1}) follows via well known (deep) theorems and tried and true techniques from a general statement:

\begin{thm*}[Monadic resolution]
Let $\Cc^\otimes$ be a presentable closed symmetric monoidal
 infinity category, let $\CAlg \Cc$ denote the category of commutative algebras in $\Cc^\otimes$. Let $A$, $B$ be algebra objects with underlying objects $\underline A$, $\underline B$.  Then
\[
\Hom_{\CAlg \Cc}(A,B) =
 \lim_{k \in \Delta} \lim_{S_1 \to \cdots \to S_k}
 \Hom_{\Cc}(\underline A^{\otimes S_1}, \underline B).
\]

\end{thm*}

\noindent
A direct approach to this general statement, however, leads to trouble. Let
\[
M: \Cc \to \Cc
\]
denote the monad associated to the adjunction between the free algebra functor and the forgetful functor and let
 $W: \CAlg \Cc \to \CAlg \Cc$ denote the comonad. By (co)monadic resolution, there's an equivalence in $\CAlg  \Cc$
\[
\colim_{k \in \Delta} W^{k+1} A \xto{\sim} A,
\]
hence an equivalence
\begin{align*}
\Hom_{\CAlg \Cc}(A,B) &=
	\lim_{k \in \Delta} \Hom_{\Cc} 
		(M^k \underline A, \underline B)
\end{align*}
in $\Ss$. We must show that if $E = \underline A$ is an object of $\Cc$ and $k \ge 1$ then 
\[
M^k E =
\underset{S = (S_1 \to \cdots \to S_{k}) \in \Fin_\gr^{k}} 
\colim
 E^{\otimes S_1}
 = \colim \Om^k_E
\]
where
\[
\Om_E^k(S_1 \to \cdots \to S_k) = E^{\otimes S_1}.
\]

Starting with the case $k=1$, we have
\[
ME = \coprod_{n \in \NN} E^{\otimes n}/ \Si_n 
= \underset{S \in \Fin_\gr} \colim \;
 E^{\otimes S}
\]
because the groupoid of sets of size $n$ is equivalent to $B\Si_n$. Turning to the case $k=2$, we begin by fixing a finite set $S$ and computing:
\begin{align*}
(\underset T \colim \; E^{\otimes T})^{\otimes S}
&=  
 \bigotimes_{s \in S} 
\underset{T_s} \colim \; E^{\otimes T_s}
\intertext{which is just a change of notation,}
\\
&=  
\underset{ (T_s)_{s\in S} } \colim \; 
\bigotimes_{s \in S} E^{ \otimes T_s} 
\intertext{by compatibility of tensor products with colimits,}
\\
&= 
\underset{ (T_s)_{s\in S} } \colim \; 
E^{ \otimes ( \coprod_{s \in S} T_s )  }.
\end{align*}
If we could make these equalities functorial in $S$, we could take colimits over $S$ to obtain 
\[
M^2E = \underset{T \to S} \colim \; E^{\otimes T}.
\]
Unfortunately, we are unable to construct such a functor directly. Instead, we write the induction step as follows:
\begin{align*}
M^{k+1}E 	&= M(M^kE) \\
		&= M \left( \colim_{\Fin^k_\gr} \Om_E^k \right) \\
		& \overset != \colim_{M(\Fin^k_\gr)} ( M\Om^k_E) \\
		&= \colim_{\Fin_\gr^{k+1}} \Om_E^{k+1}.
\end{align*}
The claimed equality (!) is best stated in an abstract setting. Let $\Cc^\otimes$ be a presentable 
 symmetric monoidal infinity category, let $\Dd$ be an infinity category, and let 
\[
f : \Dd \to \Cc
\]
be a functor to the underlying category. Let
\[
M\Dd = UF \Dd
\]
denote the underlying category of the free symmetric monoidal infinity category generated by $\Dd$. Then there's an induced functor 
\[
Mf: M\Dd \to \Cc
\]
and an equivalence 
\[
M(\colim_\Dd f) = \colim_{M\Dd} Mf
\]
in $\Cc$. To prove this, we treat $\Cc^\otimes$ as a variable, we use the adjunction between symmetric monoidal infinity categories and presentable symmetric monoidal infinity categories with functors that preserve small colimits, and we reduce to establishing the commutativity
\[
\xymatrix{
\Pp(F \Dd) \\
\Pp(FUF\Dd) \ar[u] &
\Pp(F\{*\}) \ar[l] \ar[ul]
}
\]
of a certain diagram of categories of presheaves. As it turns out, the latter is essentially controlled by 1-categories.

\segment{i3}{Proper case, motivic Chow coalgebras}

As mentioned above, when $X$ is proper and $Z = \Spec k$ is a field, we may replace the motivic complex $C^*_{\DA}(X) \in \DA(Z,\QQ)$ associated to $X$, by the Chow motive $C^*_{\ChM}(X) \in \ChM(Z,\QQ)$, an object of the additive tensor category of Chow motives over $Z$ with $\QQ$-coefficients. There is again a natural monoid structure. Actually, we prefer in this setting to work with the homological motive $C_*^{\ChM}(X)$, which then has a structure of \textit{co}monoid. By considering morphisms of Chow motives which preserve this extra structure, we obtain the category
\[
\MChC(Z,\QQ)
\]
of motivic Chow coalgebras mentioned earlier, and a functor 
\[
C_* = C_*^{\MChC}: \SmPr(Z)
 \to \MChC(Z,\QQ)
\]
from smooth proper schemes to motivic Chow coalgebras. As mentioned above, this functor factors through the functor $C_{\DMdga}^*$.

In particular, if $Y$ is also smooth and proper over $Z$, we obtain a commuting diagram
\[
\tag{$*_{Y,Ch}$}
\xymatrix{
\Hom_Z(Y,X) 
	\ar[d] \ar[r] &
	\Hom_{Z_p}(Y_p, X_p) 
	\ar[d]^-\al
	\ar@<5ex>@/^9ex/[dd]^-{\al_{Ch}}
\\
\Hom_{\DMdga(Z,\QQ)}
	(C^*X, C^*Y)
	\ar[r]_-{\la} 
	\ar[d] &
	\Hom_{\DMdga(Z_p,\QQ)}
	(C^*X_p, C^*Y_p)
	\ar[d]
\\
\Hom_{\MChC(Z,\QQ)}
	(C_*Y, C_*X)
	\ar[r]_-{\la_{Ch}} &
	\Hom_{\MChC(Z_p,\QQ)}
	( C_*Y_p, C_*X_p)
}
\]
and equation \ref{mad2b1}($**_Y$) implies that
\[
\tag{$**_{Y,Ch}$}
\Hom_Z(Y,X) = \al_{Ch} \inv (\Im \la_{Ch}).
\]
So for instance, Kim's conjecture implies this equality when $Y = Z = \Spec k$ is a number scheme and $X$ is a proper hyperbolic curve.

Our work on the case of dimension zero shows that in that case, ($**_{Y,Ch}$) holds, and then uses this result to show that in the same case \ref{mad2b1}($**_Y$) holds. In fact, in terms of the development of our work, equation ($**_{Y,Ch}$) came first, the idea to start with the case of finite \'etale $k$-schemes, and the ensuing expectation that in that case we have an equivalence of rational points and motivic points came next, and the work presented here followed. So, at least on a small, personal, scale, we've seen Kim's conjecture help to \textit{predict} a new phenomenon.

\segment{i5}{Relationship to motivic cohomology}

Although there is a priori interest in the vector spaces
\[
\Ext^i_{MM} \big(\QQ(0),  H_j(X)(k)\big)
\]
of extensions in the conjectural Tannakian category of mixed motives 
(for instance, in connection with the Beilinson conjectures, and with Chabauty's method (never mind if the `$j$' should be upper or lower)), it is the associated RHom's 
\[
\Hom_{\DA} \big( C^*X, \QQ(n)[m] \big)
\]
that are usually referred to as \emph{motivic homology} and form the bridge to K-theory. 

Let $\pi^{R}_1(Z)$ denote any of the Tannakian fundamental groups commonly considered. (If the category $MM$ were available, we would take $R=MM$; Kim essentially considers a category of Galois representations obeying certain local conditions.) Let $\pi_1^\un(X)^R$ denote the associated realization of the unipotent fundamental group of $X$. Then the abelianization of $\pi_1^\un(X)^R$ is the associated first homology group $H_1^R(X)$, and 
\[
H^1\big( \pi_1^R(Z), H_1^R(X) \big) 
	= \Ext^1_R \big(
		\QQ(0), H_1^R(X) \big).
\]
Moreover, there's a commutative diagram
\[
\xymatrix
@R=3ex
{
\mbox{Nonabelian}
	&
	\mbox{Abelian}
\\
X(Z) \ar[d] \ar@{=}[r]
	&
	X(Z) \ar[d]
\\
\Hom_{\DMdga}(C^*X, \one)
	\ar[d] \ar[r]
	&
	\Hom_{\DA}
		\big( C^*X, \QQ(0) \big)
	\ar[d]
\\
H^1\big( \pi_1^R(Z), \pi_1^\un(X)^R \big)
	\ar[r] 
	&
	\Ext^1_R \big(
		\QQ(0), H_1^R(X) \big)
}
\]
in which our factorization of Kim's \emph{unipotent Kummer map} through \emph{motivic points} maps to an analogous factorization of the abelian Kummer map.\footnote{
Indeed, from this point of view, a morphism
$
C^*X \to \QQ(0)
$
in $\DA(Z,\QQ)$ might be called a ``rational \textit{linear} motivic point''.
} So it's reasonable to say that the Ext groups of Chabauty's method stand to (abelian) motivic homology as the nonabelian cohomology groups 
\[
H^1\big( \pi_1^R(Z), \pi_1^\un(X)^R \big),
\]
which form the basis for Chabauty-Kim theory, stand to the pointed sets
\[
\Hom_{\DMdga}(C^*X, \one)
\]
considered above.\footnote{
A point which may cause confusion is that the pointed set $H^1\big( \pi_1^R(Z), \pi_1^\un(X)^R \big)$ is a \emph{co}homology set in the sense of group cohomology (that is, when considered as a functor of  $\pi_1^R(Z)$), but is covariant as a functor of $X$.
}
 It is interesting to note, that data about tensor powers of the first homology of $X$ (which sometimes gives rise to the higher Tate objects $\QQ(n)$) is enfolded into the latter, and appears in the graded quotients of the unipotent fundamental group.

\subsection*{Aknowlegements}
We heartily thank Shane Kelly for helpful conversations throughout the development of this work. We would also like to acknowledge our intellectual debt to Marc Levine; in particular, our notion of \textit{motivic Chow coalgebra} and the question of its relationship to motivic dga's grew out of conversations with him.

\section{Monadic resolution}

\segment{}{}
Let $Z$ be a separated Noetherian scheme. We let $M(Z, \QQ)$ denote the category of spectra of complexes of presheaves with \'etale-local, $\AA^1$-local, $\QQ(1)$-stable model structure, used for instance in Ayoub \cite[\S3]{AyoubEt} to construct the \emph{triangulated category
\[
\DA^\et(Z,\QQ) = \Ho M(Z,\QQ)
\]
 of \'etale $Z$-motives}. Recall from section 1 of \textit{Rational motivic path spaces} \cite{RMPS} that the category $\Mdga(Z,\QQ)$ of commutative monoid objects in $M(Z,\QQ)$ inherits a model structure; its homotopy category $\DMdga(Z,\QQ)$ is the category of \emph{motivic dga's}.
 
In this article we depart from the model categorical language of \cite{RMPS}, and rely rather heavily on the language of infinity categories, as formulated by Lurie \cite{LurieTopos, LurieHA}. There exists a stable monoidal infinity category $\DDA(Z,\QQ)$ factorizing the morphism 
\[
M(Z,\QQ) \to \DA(Z,\QQ)
\]
such that $\pi_0 \DDA(Z,\QQ) = \DA(Z,\QQ)$, as well as a monoidal infinity category $\DDMdga(Z,\QQ)$ factorizing the morphism
\[
\Mdga(Z,\QQ) \to \DMdga(Z,\QQ)
\]
such that $\pi_0 \DDMdga(Z,\QQ) = \DMdga(Z,\QQ)$. Moreover,  there is a canonical equivalence
\[
\DDMdga(Z,\QQ) =
 \CAlg \big( \DDA(Z,\QQ) \big).
\]

\segment{a1}{}
We let $\Delta$ denote the category of weakly totally ordered nonempty sets.  We denote the object $\big( \{0, \dots, n\}, {\le} \big)$ by $[n]$ ($n \ge 0$). We let
\[
\Fin^k_\gr = (\Fin^{\Delta^{k-1}})_\gr
\]
denote the category whose objects are sequences of maps of finite nonempty sets
\[
S_1 \to \cdots \to S_k.
\]
If $\Cc$ is an $\infty$-category, and $E,F$ are objects, we denote by 
$
\Hom_\Cc(E,F)
$
the \emph{space} of morphisms from $E$ to $F$, an object of the $\infty$-category $\Ss$ of spaces of \cite[Chapter I]{LurieTopos}.

\begin{sprop}
\label{monres}
Let $X$, $Y$ be smooth schemes over $Z$. Then there is an isomorphism of homotopy types
\[
\Hom_{\DDMdga}(C^*X, C^*Y) =
 \lim_{k \in \Delta} \lim_{(S_1 \to \cdots \to S_k) \in \Fin^k_\gr}
 \Hom_{\DDA}(C^*X^{\otimes S_1}, C^*Y).
\]
\end{sprop}

Proposition \ref{monres} may be taken to mean simply that there exists a diagram with the given shape and the given vertices which computes the Hom space on the left. The precise construction of the diagram will be given in the proof. In any \emph{symmetric monoidal presentable $\infty$-category}, we assume the tensor product commutes with small colimits separately in each variable. 

\begin{sthm}[Monadic resolution] 
\label{monresabs}
More generally, Let $\Cc^\otimes$ be a
  symmetric monoidal presentable $\infty$-category. Let $\CAlg \Cc$ denote the category of commutative algebras in $\Cc$. Let $A$, $B$ be algebra objects with underlying objects $\underline A$, $\underline B$. Then
\[
\Hom_{\CAlg \Cc}(A,B) =
 \lim_{k \in \Delta} \lim_{(S_1 \to \cdots \to S_k)\in \Fin^k_\gr}
 \Hom_{\Cc}(\underline A^{\otimes S_1}, \underline B).
\]
\end{sthm}

Our goal for this section is to prove theorem \ref{monresabs}.

\segment{a2}{Left Kan extensions}%
We establish notation and recall a few basic facts concerning left Kan extensions. Consider functors $p$ and $F$ between infinity-categories $\Bb, \Cc, \Dd$ as in the diagram below. 
\[
\xymatrix{
\Cc_{p/B}
\ar[r] \ar@/^3ex/[rr]^-{F_{p/B}}
&
\Cc
\ar[r]_-F \ar[d]_-p
&
\Dd
\\
\{*\}
\ar[r]_-B
&
\Bb
\ar[ur]_-{\Lan_pF}
\\
}
\]
Assume $\Dd$ is cocomplete. Then there exists a functor $\Lan_pF$ as in the diagram with the following properties. Let $B$ be an object of $\Bb$ and let $\Cc_{p/B}$ denote the infinity category of objects under $p$ over $B$. (We recall for the reader's convenience that its objects consist of pairs $(C,f)$ where $C$ is an object of $\Cc$ and
\[
f: p(C) \to B
\]
is a morphism in $\Bb$.) Let $F_{p/B}$ denote the induced functor as in the diagram. Then
\[
(\Lan_pF)(B) = \colim F_{p/B}.
\]
Moreover, we have
\[
\colim (\Lan_pF) = \colim F.
\]

\segment{j29b}{}
In the setting of theorem \ref{monresabs}, let $\Cc = U(\Cc^\otimes)$ denote the underlying infinity category of $\Cc^\otimes$ and let
\[
M = M_\Cc  = UF:  \Cc \to  \Cc
\]
denote the monad associated to the free algebra functor and the forgetful functor \cite[\S4.7]{LurieHA}. If $\Dd$ is an infinity category (or an ordinary category), we let $\Dd_\gr$ denote the infinity groupoid (or ordinary groupoid) obtained by removing all noninvertible morphisms. If $E$ is an object of $\Cc$, we let $\Om_E$ denote the functor of tensor powers of $E$
\[
\Fin_\gr \to \Cc;
\]
on vertices it is given by
\[
\Om_E(S) = E^{\otimes S}.
\]
A construction of $\Om_E$ may be extracted from the adjunction between the free algebra functor and the forgetful functor of Lurie \cite[\S3.1]{LurieHA} as follows. The object $E \in \Cc$ corresponds to a functor 
\[
E: \{ * \} \to \Cc.
\]
The adjunction then provides an associated functor
\[
F(\{*\}) \to \Cc^\otimes
\]
from the free symmetric monoidal infinity category generated by the one-point infinity category $\{*\}$. Taking underlying categories on both sides, we obtain the desired functor as the composite:
\[
\Fin_\gr = UF(*) \to U(\Cc^\otimes) = \Cc.
\]
For $n \in \NN$, let $\Om_E^n$ denote the composite
\[
\Fin^n_\gr \xto{\Pr_1} {\Fin}_\gr \xto{\Om_E} \Cc.
\]

\begin{sprop}
\label{Wit}
In the situation and the notation of segment \ref{j29b}, we have an equivalence 
\[
\tag{*}
M^n(E) \sim \colim \Om_E^n.
\]
\end{sprop}

Proposition \ref{Wit} constitutes the main technical step in our proof of theorem \ref{monresabs}; its proof spans segments \ref{j28a} -- \ref{24May2019f}. 

\segment{j28a}{}
If $\Aa^\otimes \to \Finst$ is a small symmetric monoidal infinity category, we denote by
 \[\tag{*}
\Pp(\Aa^\otimes) = \Pp(\Aa)^\otimes \to \Finst
 \]
 the symmetric monoidal presentable category of presheaves with values in the infinity category of topological spaces \cite[Corollary 4.8.1.12]{LurieTopos}.
  We recall that the monoidal structure is given by Day convolution.
  The underlying category is given by
 \[
 U\Pp(\Aa)^\otimes = \Pp(U\Aa^\otimes)
 \]
 where
 \[
 U\Aa^\otimes = \Aa^\otimes_{\langle 1 \rangle}
 \]
 denotes the underlying category of $\Aa^\otimes$.
 
If $\Cc^\otimes \to \Fin_*$ is a symmetric monoidal infinity category with underlying category
\[
\Cc = U\Cc^\otimes,
\]
we denote by
\[
\xymatrix{
\Cc 
\ar@/^1ex/[r]^-{F_\Cc}
&
**[r] \CAlg \Cc
\ar@/^1ex/[l]^-{U_\Cc}
}
\]
the free algebra and forgetful functors. When there is no danger of confusion, we drop the subscripts. In particular, if $\Dd$ is an infinity category, then $F\Dd \to \Fin_*$ denotes the free symmetric monoidal infinity category generated by $\Dd$. We let
\[
\Fin^{k \coprod}_{\gr} \to \Fin_*
\]
denote the cocartesian symmetric monoidal infinity category associated to $\Fin^k_\gr$.

\segment{v3-1}{}
If $\Dd^\otimes$ is an essentially small symmetric monoidal $\infty$-category and $\Cc^\otimes$ is a presentable symmetric monoidal $\infty$-category, then the Yonneda embedding induces an equivalence of $\infty$-categories
\[
\tag{*}
y:
\Fun^{\otimes, L}
\big( \Pp(\Dd)^\otimes, \Cc^\otimes \big)
\xto{\sim}
\Fun^{\otimes}
\big( \Dd^\otimes, \Cc^\otimes \big)
\]
which is functorial in $\Cc^\otimes, \Dd^\otimes$. The analogous statement that there's a similar equivalence
\[
y_\m{lax}: \tag{$*_\m{lax}$}
\Fun^{\m{Lax}, L}
\big( \Pp(\Dd)^\otimes, \Cc^\otimes \big)
\xto{\sim}
\Fun^{\m{Lax}}
\big( \Dd^\otimes, \Cc^\otimes \big)
\]
 with symmetric monoidal functors replaced by lax symmetric monoidal functors follows from Proposition 4.8.1.10 of \cite{LurieHA}.  The functor categories of (*) are full subcategories of the respective categories of lax symmetric monoidal functors. The functor of $y_\m{lax}$ is given by composition with the symmetric monoidal Yonneda embedding, hence carries symmetric monoidal functors to symmetric monoidal functors. This implies the existence and full faithfulness of the functor $y$. 
 
It remains to show that if a lax symmetric monoidal functor which commutes with small colimits
\[
F: \Pp(\Dd)^\otimes \to \Cc^\otimes
\]
becomes symmetric monoidal upon composition with the symmetric monoidal Yonneda embedding 
\[
\Yy^\otimes: \Dd^\otimes \to \Pp(\Dd)^\otimes,
\]
then $F$ must have been symmetric monoidal to begin with. Recall that it suffices to check that the morphisms
\[
\one \to F(\one)
\quad
\mbox{ and }
\quad
F(\zeta) \otimes F(\eta) \xto{\ka} F(\zeta \otimes \eta)
\]
are equivalences. Since the unit presheaf is represented by the unit object of $\Dd$, the compatibility with unit is immediate. For the compatibility with $\otimes$, recall that $\Pp(\Dd)$ is generated under small colimits by representable presheaves. For immediate use only, we define $\eta \in \Pp(\Dd)$ to be \emph{good} if $\ka$ is an equivalence whenever $\ze$ is representable. The class of good presheaves includes all representable presheaves. Since both $F$ and $\otimes$ preserve small colimits, the class of good presheaves is closed under small colimits. Consequently, every presheaf is good. Now, fixing a presheaf $\eta$, define a presheaf $\zeta \in \Pp(\Dd)$ to be $\eta$-\emph{good} if $\ka$ is an equivalence. We've seen that the class of $\eta$-good presheaves includes all representable presheaves. Again, the class of $\eta$-good presheaves is closed under small colimits. Hence all presheaves are $\eta$-good. This completes the verification and establishes the equivalence (*).

\begin{slm}
\label{11aug-1}
Recall that $F\{*\} \to \Finst$ denotes the free symmetric monoidal category generated by the one-point category, and that
\[
M\{*\} = UF\{*\} = \Fin_\gr
\]
is its underlying category. Let $\delta_1 \in \Pp(M\{*\})$ denote the presheaf which sends singletons to the one-point space and all other objects to the empty space. If $\Cc^\otimes \to \Finst$ is a symmetric monoidal presentable $\infty$-category, then symmetric monoidal functors which preserve small colimits
\[
\xymatrix{
\Pp(F\{*\})^\otimes \ar[rr]^f \ar[dr] && \Cc^\otimes \ar[dl]
\\
& \Finst
}
\]
are uniquely determined up to equivalence by the image of $\delta_1$ in $\Cc$. 
\end{slm}

\begin{proof}
 By the equivalences of categories of functors of \ref{v3-1}(*), a monoidal functor $f$ which preserves small colimits as in the lemma is uniquely determined by an associated monoidal functor 
\[
f': F\{*\} \to \Cc^\otimes.
\]
Moreover, 
\[
f(\delta_1) = f'(*).
\]
In turn, $f'$ is uniquely determined by the associated functor
\[
f'':\{*\} \to U\Cc
\]
and 
\[
f'(*) = f''(*).
\qedhere
\]
\end{proof}

\segment{1jula}{}

Let $\Fin^{\coprod} \to \Fin_*$ denote the cocartesian symmetric monoidal structure on the category $\Fin$ of finite sets. Recall that $U\Fin^{\coprod} = \Fin$. Hence, the obvious functor
\[
\{*\} \to \Fin
\]
gives rise by adjunction to a monoidal functor 
\[
\xymatrix{
F\{*\}
\ar[dr] \ar[rr]^i
&&
\Fin^{\coprod}
\ar[dl]
\\
&
\Fin_*.
}
\]
We have $M\{*\} = \Fin_\gr$ (the groupoid of finite sets) and $Mi$ is the inclusion of ordinary categories
\[
\Fin_\gr \subset \Fin.
\]
We thus have a symmetric monoidal functor which preserves small colimits
\[
i_!: \Pp(F\{*\}) \to \Pp(\Fin^{\coprod})
\]
and by adjunction a lax-monoidal functor
\[
i^*: \Pp(F\{*\}) \from \Pp(\Fin^{\coprod}).
\]
A simple cofinality argument shows that $i^*$ is in fact symmetric monoidal.\footnote{We alert the reader to the fact that in topos theory, our $i^*$ is usually denoted $i_*$.}

\segment{9jula}{}
Fix an essentially small infinity category $\Dd$. We will construct a diagram of symmetric monoidal presentable categories and symmetric monoidal functors which preserve small colimits as follows.  
\[
\tag{*}
\xymatrix{
\Pp \big(F(\Dd) \big)
&
\Pp(F\{*\})
\ar[l]_-{p^*}
&
\Pp(\Fin^{\coprod})
\ar[l]_-{i^*}
\\
\Pp \big( FUF(\Dd) \big)
\ar[u]^-{(\ep_{F\Dd})_! }
&
&
\Pp(F\{*\})
\ar[ll]^-{q^*} \ar[u]_-{i_!}
}
\]
The functors $i^*$, $i_!$ are those defined in segment \ref{1jula}. The functor $p^*$ is the pullback of sheaves along $F(p)$ where $p$ is the unique functor 
\[
\Dd \to \{*\}.
\]
Let
$
\ep:FU\to F
$
denote the counit of the adjunction. Then
\[
\ep_{F(\Dd)}: FUF(\Dd) \to F(\Dd)
\]
induces a functor $(\ep_{F\Dd})_!$ as shown. Finally, $q^*$ is the pullback of sheaves along $F(q)$ where $q$ denotes the unique functor
\[
UF(\Dd) \to \{*\}.
\]
(As with $i^*$, the functors $p^*$ and $q^*$ are a priori only lax-monoidal, but are in fact monoidal by a cofinality argument which we omit.)

\begin{ssprop}
\label{7aug0}
Diagram \ref{j28a}(*) commutes. 
\end{ssprop}

The proof spans segments \ref{7auga}--\ref{11augf}.

\ssegment{7auga}{}
By lemma \ref{11aug-1}, it will suffice to compute the image of the presheaf $\delta_1$ regarded as an object of (the fiber above $\langle 1 \rangle$) of the symmetric monoidal category $\Pp(F\{*\})$. In doing so, we may restrict the diagram to its fiber above $\langle 1 \rangle$. 

We begin by computing $i_! \delta_1$.

\begin{sslm}
\label{7augb}
Let $S$ be a finite set. Then $(i_!\delta_1)(S)$ is contractible. 
\end{sslm}

\begin{proof}
Let $Ui$ denote the restriction of the functor $i$ to the fibers above $\langle 1 \rangle$, and similarly for any functor between symmetric monoidal categories. There is an equivalence $(Ui)^* = U(i^*)$. Moreover, $U$ preserves adjunctions. Hence we have an equivalence $(Ui)_! = U(i_!)$. Let's denote $U(i)$ by $j$. Then we have for any presheaf $\Ff$ as in the following diagram, $j_! \Ff = \Lan_{j\op} \Ff$:
\[
\xymatrix{
\Fin\op_\gr
 \ar[r]^{\Ff} \ar @{}|\bigcap _{j\op} [d]
&
\Ss 
\\
\Fin\op
\ar @{.>}[ur]_{j_!\Ff = \Lan_{j\op}\Ff.}
}
\]
Then $(\Fun\op_\gr)_{j\op/S} = (\Fin_{S/})\op_\gr$ (an equality of 1-categories) and for any object $S \to T$ of $(\Fin_{S/})\op_\gr$, we have
\[
(\de_1)_{j\op/S}
(S \to T) = \de_1(T).
\]
In computing the colimit, we may ignore empty topological spaces, and we are left with a colimit over $\{*\}$ taking the value $\{*\}$. Hence 
\[
(i_!\delta_1)(S) 
= \colim \big( (\de_1)_{j\op/S} \big)
= \{*\}
\]
as hoped.
\end{proof}

\ssegment{}{}
Let $\underline *$ denote the terminal presheaf. Then by lemma \ref{7augb}, the unique morphism
\[
i_!\de_1 \to \underline *
\]
is an equivalence. Pulling back along $i$ and $p$ (restricted to the fiber above $\langle 1 \rangle$), we conclude that
\[
p^*i^*i_!(\de_1) \cong \underline *
\]
is equivalent to the terminal presheaf on $M(\Dd)$.

\ssegment{11auga}{}
Turning our attention to $q^*$, we find that $q^*\de_1 = \de_{M(\Dd)}$ is the presheaf on $M^2\Dd$ which sends objects of $M(\Dd)$ to the singleton space and all other objects to the empty space. To see this, let us denote $M(\Dd)$ by $\Ee$ and note that the restriction of the map $q^*$ defined in segment \ref{9jula} to the fiber above $\langle 1 \rangle$ is the pullback functor along $M(q)$. The latter breaks up as a sequence of functors
\[
M_n : \Ee^{\times n} / \Si_n \to B\Si_n
\]
hence $q^*\de_1 = 
\de_{\Ee^{\times 1}/\Si_1} = \de_\Ee$ as claimed.

\begin{sslm}
\label{11auge}
Let $\Ee \to \Fin_*$ be a symmetric monoidal $\infty$-category, let $\ep_\Ee$ denote the functor
\[
FU\Ee \to \Ee
\]
induced by the counit of the adjunction, let $\ep_{\Ee!}$ denote the covariant operation on presheavs of spaces restricted to a functor 
\[
\Pp(UFU\Ee) \to \Pp(U\Ee),
\]
let
$
\de_{U\Ee} \in \Pp(UFU\Ee) 
$
denote the characteristic presheaf of $U\Ee \subset UFU\Ee$, and let $X$ be an object of $U\Ee$. Then
$
\ep_{\Ee!}\de_{U\Ee}(X)
$
is contractible. 
\end{sslm}

\begin{proof}
As in lemma \ref{7augb}, the lower-shriek operation is given by a left Kan extension (we drop the superscripts ``op''):
\[
\xymatrix{
(UFU\Ee)_{U(\ep_\Ee)/X}
\ar[r]
\ar@/^3ex/[rr]^-{(\de_{U\Ee})_{U(\ep_\Ee)/X}}
&
UFU\Ee 
\ar[r]_-{\de_{U\Ee}}
\ar[d]_{U(\ep_\Ee)}
&
\Ss
\\
{*}
\ar[r]_{X}
&
U\Ee
\ar@{.>}[ur]_{\ep_{\Ee!}\de_{U\Ee}}
}
\]
in computing its value
\[
\ep_{\Ee!}\de_{U\Ee}(X) = 
\colim (\de_{U\Ee})_{U(\ep_\Ee)/X}
\]
we may restrict attention to the full subcategory on which the functor $\ep_{\Ee!}\de_{U\Ee}$ takes nonempty values. The decomposition of $UFU\Ee$ as a disjoint union of symmetric powers of $U\Ee$ shows that $U(\ep_\Ee)$ restricted to $\de_{U\Ee}\inv\{*\} = U\Ee$ is simply the identity functor. Consequently, 
\[
(\de_{U\Ee})_{U(\ep_\Ee)/X} \inv\{*\}
\]
is equal to the overcategory $U\Ee_{/X}$. The colimit of the constant functor $\underline *$ is the geometric realization of the source. The geometric realization of a category possessing a terminal object is contractible. 
\end{proof}

\ssegment{11augf}{}
Applying lemma \ref{11auge} to $\Ee = F\Dd$, we find that 
\[
\ep_{F\Dd!}\de_{M(\Dd)} \cong \underline *,
\]
which completes the proof of proposition \ref{7aug0}.

\segment{j29a}{}
We now use the commutativity of diagram \ref{9jula}(*) to prove proposition \ref{Wit}. Taking categories of symmetric monoidal functors which preserve small colimits
\[
\Fun^{L, \otimes}(?, \Cc^\otimes) 
\]
into our symmetric monoidal presentable $\infty$-category $\Cc^\otimes \to \Fin*$, we obtain a diagram like so
\[
\tag{$*_\Cc^\otimes$}
\xymatrix
@C=1ex
{
\Fun\big(\Pp \big(F(\Dd) \big), \Cc^\otimes \big)
&
\Fun\big( \Pp(F\{*\}), \Cc^\otimes \big)
\ar@{<-}[l]
&
\Fun\big( \Pp(\Fin), \Cc^\otimes \big)
\ar@{<-}[l]
\\
\Fun\big( \Pp \big( FUF(\Dd) \big), \Cc^\otimes \big)
\ar@{<-}[u]
&
&
\Fun\big( \Pp(F\{*\}), \Cc^\otimes \big)
\ar@{<-}[ll] \ar@{<-}[u]
}
\]
which commutes up to homotopy. In turn, by segment \ref{v3-1}, diagram ($*^\otimes_\Cc$) gives rise to a diagram of $\infty$-categories of functors like so
\[
\tag{$*_\Cc$}
\xymatrix{
\Fun(\Dd , U\Cc)
\ar[r]^-{\colim} 
\ar@{=}[d]
&
U\Cc
\ar[r]^-F
&
\CAlg \Cc
\ar[dd]^-U
\\
\Fun^{L, \otimes}(F\Dd, \Cc^\otimes)
\ar[d]_-U
\\
\Fun(M\Dd, U\Cc)
\ar[rr]_-\colim
&&
U\Cc
}
\]
which commutes up to homotopy. We have abbreviated $U(\Cc^\otimes)$ by $U\Cc$. 

\segment{j29d}{}
Hence, if $\Dd$ is an essentially small infinity category, $\Cc^\otimes$ is a symmetric monoidal presentable infinity category, and $f:\Dd \to U\Cc$ is a functor to the underlying category
$
U\Cc 
$
of $\Cc^\otimes$, there's an induced functor
\[
Mf: M\Dd \to U\Cc
\]
and
\[
\tag{*}
\colim_{M\Dd} Mf = M(\colim f). 
\]

\segment{j29c}{}
We have $M\{*\} = \Fin_\gr$, the groupoid of finite sets and isomorphisms, and more generally, 
\[
\tag{*}
M  (\Fin^k_\gr)  = \Fin^{k+1}_\gr.
\]
Moreover,
\[
\tag{**}
M(\Om^k_E) = \Om^{k+1}_E.
\]

\ssegment{24May2019a}{}
To establish \ref{j29c}(*), recall that if $X \in \Ss$ is an object of the infinity category $\Ss$ of spaces (i.e. a simplicial set in which every horn is fillable) and $f: X\to \Ss$ is a functor, then the Grothendieck construction on $f$ is equivalent to the colimit of $f$. We regard $\Fin^k_\gr$ as an object of $\Ss$. The Cartesian symmetric monoidal structure 
\[
\Ss = U\Ss^{\prod}
\]
gives rise to a functor 
\[
\Om_{\Fin^k_\gr}: \Fin_\gr \to \Ss
\]
\[
T \mapsto (\Fin^k_\gr)^{\prod T}.
\]
The Grothendieck construction on $\Om_{\Fin^k_\gr}$ is equivalent to the last projection
\[
\Pr_{k+1}: \Fin^{k+1}_\gr \to \Fin_\gr 
\]
\[
(S_1 \to \cdots \to S_{k+1}) \mapsto S_{k+1}.
\]
Hence, 
\begin{align*}
M(\Fin^k_\gr) 
	= \colim \Om_{\Fin^k_\gr} 
	= \Fin_\gr^{k+1}.
\end{align*}

\ssegment{24May19b}{}
To establish \ref{j29c}(**), note first that if 
\[
F^\otimes: \Aa^\otimes \to \Bb^\otimes
\]
is a symmetric monoidal functor between symmetric monoidal $\infty$-categories with underlying functor 
\[
F: \Aa \to \Bb
\]
and if 
\[
\xymatrix{
\Ee \ar[d]_G \ar[dr]^H \\
\Aa \ar[r]_F &
\Bb
}
\]
is a commuting triangle of $\infty$-categories, then the functoriality of the adjunction between the free and forgetful functors gives rise to a commuting triangle 
\[
\xymatrix{
F\Ee \ar[d]_G \ar[dr]^H \\
\Aa^\otimes \ar[r]_{F^\otimes} &
\Bb^\otimes.
}
\]

\ssegment{24May2019c}{}
Disjoint union of sets induces a symmetric monoidal structure on $\Fin_\gr$ which may be obtained as a restriction of the cocartesian structure on the full category of finite sets. The functor $\Om_*$ associated to the one-point set is simply the identity functor on $\Fin_\gr$. Thus, 
\[
\Om^k_* = \Pr_1 : \Fin^k_\gr \to \Fin_\gr
\]
and the triangle
\[
\xymatrix{
\Fin^k_\gr \ar[d]_{\Om_*^k} \ar[dr]^{\Om^k_E} \\
\Fin_\gr \ar[r]_{\Om_E} &
\Cc
}
\]
commutes. The bottom arrow is by construction the underlying functor of a symmetric monoidal functor 
\[
F\{*\} \to \Cc^\otimes.
\]
Taking adjoints as in segment \ref{24May19b} and then applying the functor $U$ of underlying categories and functors, we find that 
\[
M(\Om^k_E) = \Om_E \circ M(\Om^k_*).
\]

\ssegment{}{}
We are thus reduced to establishing a natural isomorphism of functors of 1-categories 
\[
M(\Om^k_*) = \Om^{k+1}_*,
\]
which is comparatively straightforward. We nevertheless work this out. The following observation applies equally (if less precisely) to infinity categories: if
\[
f: \Aa \to \Bb = U\Bb^\otimes
\]
is a functor to the underlying category of a symmetric monoidal category, then
\[
M \Aa = \coprod_n \Aa^{\prod n} / \Si_n
\]
and the functor 
\[
Mf: M\Aa \to \Bb
\]
is induced by the functors 
\[
\Aa^{\prod n} \to \Bb
\]
\[
(x_1, \dots, x_n) \mapsto 
f(x_1) \otimes \cdots \otimes f(x_n).
\]
Applying this description of $Mf$ to $f = \Om^k_*$, we find that the triangle 
\[
\xymatrix{
\coprod_n(\Fin^k_\gr)^{\prod n} / \Si_n  
\ar@{=}[d] \ar[dr]^{M\Om_*^k}\\
\Fin^{k+1}_\gr 
\ar[r]_{\Om_*^{k+1}} &
\Fin_* 
}
\]
\[
\xymatrix{
(S_\bullet^1, \dots, S_\bullet^n)
\ar@{|->}[dr] 
\ar@{|->}[d]\\
(\coprod_i S_1^i \to \cdots \to 
\ar@{|->}[r]
\coprod_i S_k^i \to \{1, \dots, n\}) &
\coprod_{i=1}^n S_1^i
}
\]
commutes indeed.

\segment{24May2019f}{}
Hence, applying equation \ref{j29d}(*) repeatedly, we obtain equation \ref{Wit}(*), which completes the proof of proposition \ref{Wit}.

\segment{monresprf}{Proof of theorem \ref{monresabs}}
We switch to using underline to denote the underlying object of an algebra-object. Let
\[
W: \CAlg \Cc \to \CAlg \Cc
\]
denote the comonad associated to the adjunction between the free algebra functor and the forgetful functor. 
By monadic resolution \cite[4.7.3.14]{LurieHA}, there's an equivalence in $\CAlg  \Cc$
\[
\colim_{k \in \Delta} W^{k+1} A \xto{\sim} A.
\]
Thus,
\begin{align*}
\Hom_{\CAlg \Cc}(A,B) &=
	\lim_{k \in \Delta} \Hom_{\Cc} 
		(M^k \underline A, \underline B)
\\
	&= \lim_{k \in \Delta} \Hom_{\Cc} 
		(\colim \Om_{\underline A}^k, \underline B)
\\
	&= \lim_{k \in \Delta}
		\lim_{(S_1 \to \cdots \to S_k) \in \Fin^k_\gr}
		 \Hom_{\Cc} 
		(\underline A^{\otimes S_1}
			, \underline B),
\end{align*}
which completes the proof of the monadic resolution theorem (\ref{monresabs}).

\section{Main theorem}

Having formed a bridge from our spaces of motivic morphisms to spaces of morphisms in $\DDA$ in our \textit{monadic resolution theorem}, we now form a bridge from the latter to K-theory. This is essentially a matter of assembling several (deep) theorems from the literature on motives.

\segment{}{Overview}
We begin with a brief overview of the ingredients to be assembled. The situation we have in mind is as follows. We consider a separated Noetherian base scheme $Z$ and a smooth scheme $X$ over $Z$. We assume $X$ admits a compactification $\overline X$ which is smooth over $Z$, such that the complement $V$ is a relative snc divisor. If $S$ is a finite set, we wish to write the cohomological motive $C^*(X^S)$ of $X^S$ in $ \DDA(Z,\QQ)$ as a limit of homological motives of smooth schemes. Let $V_S$ denote the complement of $X^S$ in $\overline X^S$. By poincar\'e duality,
\[
C^*(X^S) = C_*^c(X^S).
\]
The symbol $C_*^c$ refers to \textit{homology with compact supports}, and may be defined as a (homotopy) cofiber
\[
\xymatrix{
 C_*^c(X^S)
 	&
 	0 
	\ar[l]
\\
C_*(\Xbar^S)
	\ar[u]
	&
	C_*(V_S).
	\ar[l] \ar[u]
}
\]
In turn, the motive $C_*(V_S)$ may be written as the (homotopy) colimit of the hypercube of motives of intersections of components, as we now recall.

\segment{hyp}{Combinatorics of hypercubes}
We begin with a brief recap of the combinatorics of hypercubes; this is not necessary for what follows. Recall that $\Delta^1$ denotes the category with two objects and one non-identity morphism
\[
0 \xto{\tau} 1.
\]
The notion of \emph{hypercube} we have in mind is most naturally defined as the Cartesian power $(\Delta^1)^S$ of the category $\Delta^1$ by a finite set $S$. For example, if $S = \{1,2\}$ has two elements, then the objects of $(\Delta^1)^S$ are the ordered pairs of elements of $\{0,1\}$, and the category $(\Delta^1)^S$ may be depicted as a square
\[
\xymatrix{
(0,0) \ar[r] \ar[d] & (0,1) \ar[d] 
\\
(1,0) \ar[r] & (1,1).
}
\]
In general, the objects of $(\Delta^1)^S$ are the maps 
\[
f: S \to \{0,1\}
\]
and the morphisms are given by
\[
\Hom(f,f') = \prod_{s \in S}
	\Hom \big( f(s), f'(s) \big).
\]

There is a slightly different construction of the categories $(\Delta^1)^S$ which we now describe. Let $PS$ denote the category whose objects are the subsets of $S$ and whose morphisms are the inclusions. Then there's an isomorphism of categories 
\[
\Psi: PS \xto{\sim}  (\Delta^1)^S
\]
given on objects by
\[
\Psi(S') = 1_{S'}
\]
where 
\[
1_{S'} (s) = 
\left\{
\begin{matrix}
1 & \mbox{if } s \in S' 
\\
0 & \mbox{if } s \notin S'
\end{matrix}
\right.
\]
is the characteristic function of $S'$, and given on morphisms as follows. If $\iota: S'' \subset S'$ is a morphism in $PS$, then $\Psi(\iota)$ is the morphism 
\[
1_{S''} \to 1_{S'}
\]
given by
\[
\Psi(\iota) =
\left\{
\begin{matrix} 
\Id_1 & \mbox{if } s \in S'' 
\\
\tau & \mbox{if } s \in S' \setminus S''
\\
\Id_0 & \mbox{if } s \notin S'.
\end{matrix}
\right.
\]
Under this equivalence, the full subcategory of $(\Delta^1)^S$ given by removing the final vertex $1_S$ corresponds to the category $P'S$ of proper subsets of $S$.

\segment{hyp1}{Hypercube resolution and hypercube compactification}
Recall that $P'S$ denotes the category whose objects are the proper subsets of $S$ and whose morphisms are the inclusions. If $W$ over $Z$ is a union of smooth schemes $W_s$  with simple normal crossings indexed by $s$ in a finite set $S$
\[
W = \bigcup_{s \in S} W_s,
\]
we define an associated punctured hypercube of smooth $Z$-schemes
\[
H: P'S \to \Sm(Z)
\]
by the formula
\[
H(S') = \bigcap_{s\in S'} W_s.
\]
By $h$-descent, 
\[
\tag{*}
C_*(W) = \colim C_*H,
\]
where the (homotopy) colimit is take in the infinity category $\DDA(Z,\QQ)$. Indeed, this may be taken as a definition, justified, for instance, by chapter 5 of Cinsinski-D\'eglise \cite{CisDegEt}. There, the authors construct a model category, as well as a functor 
\[
C_*': \{\mbox{finite type $Z$-schemes} \}
\to \DDA_h(Z,\QQ)
\]
to its homotopy category, with the following properties. On the one hand, there's a Quillen equivalence 
\[
\tag{**}
\DD A(Z,\QQ) \cong \DD A_h(Z,\QQ)
\]
compatible with the homological motives functors $C_*$, $C_*'$. On the other hand, as a direct consequence of the construction, when $S = \{1,2\}$ has two elements, 
\[
C'_*(W) = C'_*(W_1) \coprod_{C'_*(W_{1,2})}^{(ho)} C'_*(W_2).
\]
By repeated application of the Fubini theorem, and by translating the result along the Quillen equivalence (**), we obtain equation (*).

Suppose now that $X$ over $Z$ admits an open immersion in a smooth proper scheme $\Xbar$ so that the complement $W = \bigcup_{s \in S} W_s$ is a simple normal crossings divisor. Then $C_*^c(X)$ is a colimit of colimits
\[
\xymatrix{
 C_*^c(X)
 	&
 	0 
	\ar[l]
\\
C_*(\Xbar)
	\ar[u]
	&
	\colim C_*H.
	\ar[l] \ar[u]
}
\]
We may use the Fubini theorem to assemble the colimits. This leads to a category $\KK S$ obtained from $P'S$ by adding two objects $l$ and $u$, and morphisms
\[
\xymatrix{
	&
	u 
\\
l
	&
	S' \ar[u] \ar[l]
}
\]
for each object $S'$ of $P'S$.

Recall that $M(Z,\QQ)$ denotes the model category of complexes of presheaves of $\QQ$-vector spaces on $\Sm_Z$ with \'etale local, $\AA^1$-local, model structure, stabilized for $\QQ(-1)$, and that there's a natural functor of infinity categories
\[
M(Z,\QQ)  \to  \DDA(Z,\QQ).
\]
If $\iota: X \subset \Xbar$ denotes the inclusion, we define the \emph{(motivic) hypercube compactification of $X$ associated to $\iota$} to be the functor
\[
\ka(\iota): \KK S \to M(Z,\QQ) \to \DDA(Z,\QQ)
\]
given by
\[
\ka(\iota)(l) = C_*\left(\Xbar \right),
\]
by
\[
\ka \big(\iota(S') \big) = 
	C_* \left( \bigcap_{s \in S'} W_s \right)
\]
for $S' \in P'S$, and by
\[
\ka(\iota)(u) = 0.
\]
(We note that the value on $u$ is the only part of the diagram which doesn't factor through $\Sm(Z)$. )

Combining these constructions with Poincar\'e duality, we obtain

\begin{ssprop}
\label{C*}
Let $X$ be a smooth scheme over $Z$ of relative dimension $d$. Suppose $X$ admits an open immersion $\iota: X \subset \Xbar$ in a smooth proper scheme $\Xbar$ so that the complement is a snc divisor over $Z$ (i.e. whose irreducible components are smooth over $Z$). Let $\ka(\iota)$ denote the hypercube compactification of $X$ associated to $\iota$. Then we have
\[
C^*_{\DA}(X) \sim \colim \ka(\iota)(-d)[-2d].
\]
where the (homotopy) colimit is taken in the infinity category $\DDA(Z,\QQ)$, and $\sim$ denotes homotopy equivalence, that is, isomorphism in the homotopy category $\DA(Z,\QQ)$.
\end{ssprop}

\begin{proof}
We have
\begin{align*}
C^*(X) &\sim
	C_*^c(X)(-d)[-2d] \\
\intertext{by Poincar\'e duality \cite[A.5.2(8), A.5.1(6) in the version of 2012]{CisDeg}}
	&\sim \colim \ka(\iota)(-d)[-2d]
\end{align*}
by segment \ref{hyp1}.

\end{proof}

\segment{ad1}{Adams pieces of rational K-spectra}
We now recall the link to K-theory. 
We define $\Kk^{(i)} := \QQ(i)[2i]$. Then we have on the one hand a decomposition of the rational K-theory spectrum in $\DA(Z,\QQ)$
\[
K_\QQ = \bigoplus_{i \in \ZZ} \Kk^{(i)},
\]
and on the other hand, for $X$ smooth over $Z$ and $j \in \ZZ$, we have canonical isomorphisms of $\QQ$-\textit{vector spaces}
\begin{align*}
\pi_j(\Kk^{(i)}(X)) &=
	\pi_j \Hom_{\DDA(Z,\QQ)} (C_*X, \QQ(i)[2i])
\\
	&= \Hom_{\DA(Z,\QQ)}(C_*X, \QQ(i)[2i-j])
\\
	&= K^{(i)}_j(X),
\end{align*}
the $i$th Adams piece of the rational K-group $K_j(X) \otimes \QQ$. These facts are summarized in Cisinski-D\'eglise \cite[\S14,16]{CisDeg} and go back to work of Ayoub, Morel and Voevodsky.
 
\begin{sprop}
\label{mosc1}
Let $X,Y$ be smooth schemes over $Z$. Suppose $X$ admits an open immersion 
\[
\iota: X \subset \overline X
\]
into a smooth proper scheme over $Z$ with snc divisor complement
\[
W = \bigcup_{i \in I} W_i.
\]
Let $d$ denote the relative dimension of $X$ over $Z$. Recall that
\[
\ka(\iota): \KK I \to \DDA(Z,\QQ)
\]
denotes the hypercube compactification of $X$ associated to $\iota$ (\ref{hyp1}). Let $\ka(\iota \times Y)$ denote the composite functor 
\[
 \KK I \xto{\ka(\iota)} \DDA(Z,\QQ)
 	\xto{\otimes C_*Y} \DDA(Z,\QQ),
\]
and let $\Kk^{(d)}\big( \ka(\iota \times Y) \big)$ denote the induced functor to spectra 
\[
\KK I \op 
	\xto{\ka(\iota \times Y)\op}
	\DDA(Z,\QQ)\op
	\xto{\Hom(\cdot, \QQ(d)[2d])}
	\mbox{Spectra}.
\]
Then there is an isomorphism of (stable, rational) homotopy types 
\[
\Hom_{\DDA(Z,\QQ)}(C^*X, C^*Y) =
	 \lim \Kk^{(d)} \big( \ka (\iota \times Y) \big).
\]
In particular, if $X$ itself is proper, then
\[
\Hom_{\DDA(Z,\QQ)}(C^*X, C^*Y) =  \Kk^{(d)}(X \times Y).
\]
\end{sprop}

\begin{proof}
It is a fundamental fact of infinity categories that homotopy equivalent objects (i.e. which become equivalent in the homotopy category) give rise to homotopy equivalent Hom-spaces. 
 So the computation of
\[
\Hom_{\DDA}(C^*X, C^*Y)
\]
may be carried out with the help of the homotopy equivalences recorded above. 

In particular, we have isomorphisms of homotopy types
\begin{align*}
\Hom_{\DDA}(C^*X, C^*Y) &=
	\Hom_{\DDA} 
		\big( C^*X \otimes C_*Y, \QQ(0) \big)
\\
	&= \Hom_{\DDA}
		\big( 
		\colim \kappa(\iota) \otimes C_*Y, \QQ(d) [2d]
		\big)
\\	
	&= \Hom_{\DDA}
		\big( 
		\colim \kappa(\iota\times Y), \QQ(d) [2d]
		\big)
\\
	&= \lim \Hom_{\DDA}
		\big( 
		\kappa(\iota \times Y), \QQ(d)[2d] 
		\big) 
\\
	&= \lim \Kk^{(d)} \big(
		\ka(\iota \times Y)
		\big)
\end{align*}
as claimed.
\end{proof}

\begin{sthm}[Comparison with K-theory]
\label{nmc}
Let $X,Y$ be smooth schemes over $Z$. Suppose $X$ admits an open immersion 
$
\iota: X \subset \overline X
$
into a smooth proper scheme over $Z$ so that the complement is a relative snc divisor over $Z$. Let $d$ denote the relative dimension of $X$ over $Z$. For any finite set $T$, we let $\iota^T$ denote the inclusion
\[
X^T \subset \Xbar^T.
\]
Recall that
$
\ka(\iota^T)
$
denotes the motivic hypercube compactification of $X^T$ associated to $\iota^T$ (\ref{hyp1}), a diagram in the infinity category $\DDA(Z,\QQ)$. Recall that $\ka(\iota^T \times Y)$ denotes the same diagram crossed with $Y$ (\ref{mosc1}). Recall that for $e \in \NN$,
\[
\Kk^{(e)} \big( 
	\ka(\iota^T \times Y)
	\big)
\]
denotes the associated diagram of Adams pieces of $K$-spectra (\ref{mosc1}). Then we have an isomorphism of homotopy types
\[
\tag{*}
\Hom_{\DDMdga}(C^*X, C^*Y) =
 \lim_{ k \in \Delta}
 \Omega^\infty
  \lim_{S_1 \to \cdots \to S_k}
  \Big( \lim
  \Kk^{(ds_1)} \big(
  \kappa
  (\iota^{S_1}\times Y)
  \big)
  \Big).
 \]
 In particular, if $X$ is proper, then
 \[
\Hom_{\DDMdga}(C^*X, C^*Y) =
 \lim_{k \in \Delta}
 \Omega^\infty
  \lim_{S_1 \to \cdots \to S_k}
  \big(  \Kk^{(ds_1)} (X^{S_1}\times Y) \big).
 \]
\end{sthm}

\begin{proof}
Formula (*) combines proposition \ref{monres} with proposition \ref{mosc1}.
\end{proof}

\segment{mar7}{Remark}
The outermost limits above (in both cases) involve a $\Delta$-shaped diagram
\[
\tilde \mu: N(\Delta) \to \Ss
\]
in the infinity category of spaces. Here $N(-)$ denotes the simplicial nerve of a category. We describe the associated functor 
\[
N(\Delta) \to \Ss \to \Hh
\]
to the homotopy category, which we denote by $\mu$.
We restrict attention to the case of $X$ proper, $Y = Z$.

On the level of objects, we then have
\[
\mu(k) = 
 \lim_{S_1 \to \cdots \to S_k}
  \big(  \Kk^{(ds_1)} (X^{S_1}) \big).
\]
For $n$ a natural number, we write $[n] := \{ 0, \dots, n\}$.
Now consider a fixed but arbitrary natural number $k$ and consider the map 
\[
s: [k] \to [k+1]
\]
given by
\[
s(i) := i+1.
\]
Then $\mu(s)$ is a map of homotopy types as shown in the top row of the following diagram.
\[
\xymatrix
{
\underset{S = (S_1 \to \cdots \to S_k)}
 \colim
  \big(  \Kk^{(ds_1)} (X^{S_1}) \big)
  \ar[r]
   \ar[d]_-{\Pr_{S' = (S_2 \to \cdots \to S_{k+1})}}
  &
 \underset{S = (S_1 \to \cdots \to S_{k+1})}
 \colim
  \big(  \Kk^{(ds_1)} (X^{S_1}) \big)  
  \ar[d]^-{\Pr_{S = (S_1 \to \cdots \to S_{k+1})}}
\\
  \Kk^{(ds_2)} (X^{S_2}) 
 \ar@{.>}[r]_{\phi^*}
	&
 \Kk^{(ds_1)} (X^{S_1}) 
}
\]
In fact, the map $\mu(s)$ is compatible with the projection maps and is given, for each $S$, by a morphism $\phi^*$ of K-spectra as shown in the bottom row. We now describe its motivic origin.

Let $\phi$ denote the map $S_2 \from S_1$ in $S$. Then $\phi$ induces a map of schemes 
\[
X^{S_2} \to X^{S_1}
\]
hence a morphism in $\DDA$
\[
C^*(X^{S_2}) \from C^*(X^{S_1}).
\]
By Poincar\'e duality, we have a homotopy equivalence
\[
C^*(X^{S_1}) \sim C_*(X^{S_1})(-ds_1)[-2ds_1]
\]
and similarly for $C^*(X^{S_2})$. Composing and taking hom into $\QQ(0)$ we obtain a morphism 
\begin{align*}
\Kk^{(ds_2)} (X^{S_2}) \cong 
	&
	\Hom_{\DDA} \big( C_*(X^{S_2}) ,  \QQ(ds_2)[2ds_2] \big)
\\
	& \to
	\Hom_{\DDA} \big( C_*(X^{S_1}) , \QQ(ds_1)[2ds_1] \big)
	\cong
	\Kk^{(ds_1)} (X^{S_1})
\end{align*}
as desired. 

All other maps are simply projections.

\segment{3sepa}{}
In the situation and the notation of our \textit{comparison} theorem (\ref{nmc}), let
\[
b:Y \to X
\]
be a morphism of $Z$-schemes and denote also by $b$ the associated point of the topological space
\[
\Hom_{\DDMdga}(C^*X, C^*Y).
\]
We now indicate how the construction of Bousfield \cite[\S2]{BousfieldSS} applies to give a spectral sequence which computes the homotopy groups 
\[
\pi_i 
\left(
\Hom_{\DDMdga}(C^*X, C^*Y)
,
b
\right),
\]
the set of greatest interest
\[
\Hom_{\DMdga}(C^*X, C^*Y)
=
\pi_0 
\left(
\Hom_{\DDMdga}(C^*X, C^*Y)
\right)
\]
among them.

\ssegment{coal}{}
The functor $\tilde \mu: N(\Delta) \to \Ss$ lifts to a functor 
\[
N(\Delta) \to N(\m{sSet})
\]
to the nerve of the ordinary category of simplicial sets.The latter is equivalent, in turn, to a functor
\[
M^\bullet:\Delta \to \m{sSet}
\]
between ordinary categories. 

\ssegment{ath1}{}
Let $T^\bullet_*$ be a fibrant cosimplicial simplicial set with simplicial index `$*$' and cosimplicial index `$\bullet$'. Let $T^\bullet$ denote the associated functor 
\[
\Delta \to \m{sSet}
\]
to the (ordinary) category of simplicial sets, which we may regard further as a functor 
\[
N(T^\bullet): N(\Delta) \to N(\m{sSet}) \to \Ss
\]
to the infinity category of spaces. We set
\[
\Tot T^\bullet = \lim N(T^\bullet)
\]
where the limit is taken in $\Ss$. For each $k \in \NN$, we define
\[
\cosk_k T^\bullet: \Delta \to \m{sSet}
\]
by
\[
[n] \mapsto \cosk_k(T^n).
\]
We again have an associated functor $N(\cosk_k T^\bullet)$ to the infinity category of spaces, and we set 
\[
\Tot_k T^\bullet : = \lim N(\cosk_k T^\bullet).
\]

\ssegment{cobalt}{}
Fix a point $b\in \Tot T^\bullet$. There's a natural way to associate to $b$ a vertex $b_m \in T^m$ for each $m$ so as to make 
\[
\pi_i(T^\bullet, b) := \{ \pi_i(T^m, b_m) \}_{m \in \NN}
\]
into a ``cosimplicial group''; see \cite[\S2.1]{BousfieldSS} for details. In turn, given a cosimplicial group $G^\bullet$, Bousfield constructs \emph{cohomotopy groups} $\pi^j(G^\bullet)$. There is then a homotopy (or \emph{fringe}) spectral sequence with 2\textit{nd} page 
\[
E_2^{s,t}(T^\bullet, b) = \pi^s \pi_t(T^\bullet, b)
\]
and final page 
\[
E_\infty^{t-s} = \pi_{t-s}(\Tot T^\bullet, b)
\]
\cite[\S2.4]{BousfieldSS}.

\ssegment{lead}{}
We apply the above construction to the cosimplicial simplicial set $M^\bullet$ of segment \ref{coal} and to the base point $b$ of segment \ref{3sepa}. Recall that
\[
M^k=
\Omega^\infty
  \lim_{S = (S_1 \to \cdots \to S_k)}
  \Big( \lim
  \Kk^{(ds_1)} \big(
  \kappa
  (\iota^{S_1}\times Y)
  \big)
  \Big),
\]
and that in particular, when $X$ is proper over $Z$ we have
\[
M^k = 
\Omega^\infty
  \lim_{S = (S_1 \to \cdots \to S_k)}
  \Kk^{(ds_1)} (X^{S_1}\times Y) .
\]
We obtain a homotopy spectral sequence
\[
E_2^{s,t} =
\pi^s \pi_t (M^\bullet, b) 
\Rightarrow
 \pi_{t-s} 
 \left(
 \Hom_{\DDMdga}(C^*X, C^*Y)
 ,
 b
 \right).
\]

\section{Motivic Chow coalgebras and the case of finite \'etale $k$-algebras}

\segment{Mercury}{}
Let $Z = \Spec k$, $k$ a field. We consider the category $\CAlg \DA(Z, \QQ)$ of commutative monoid objects in the homotopy category $\DA(Z,\QQ)$ of $\DDA(Z,\QQ)$. There's a natural functor 
\[
\CAlg \DDA(Z,\QQ) \to \CAlg \DA(Z,\QQ).
\]

\segment{15augb}{}
Let $\DDA^c(Z,\QQ)$, $\DA^c(Z,\QQ)$ denote subcategories of compact objects. These categories admit strong-duals. We denote the duality functor by $D$. The functor $D$ induces an equivalence 
\[
D: \CAlg \DA^c(Z,\QQ) \to \CCoalg \DA^c(Z,\QQ)
\]
from commutative monoid objects to commutative comonoid objects. 

\segment{mchc}{}
If $X$, $Y$ are smooth proper over $Z$, then by \cite[4.2.6]{VoevTriCat} we have
\[
\Hom_{\ChM(Z,\QQ)}(C_*X, C_*Y) 
= \Ch_{\dim X}(X \times Y),
\]
the Chow group of cycles of dimension equal to the dimension of $X$. Moreover, under these isomorphisms, composition of morphisms corresponds to composition of correspondences \cite{VoevChow}. We let $\ChM(Z, \QQ)$ denote the Karoubian envelope of the full subcategory of $\DA^c(Z,\QQ)$ consisting of Tate twists of homological motives of smooth proper schemes. We let
\[
\MChC(Z) := \CCoalg \ChM(Z,\QQ)
\]
denote the category of commutative comonoids. 

\segment{15auga}{}
Let $\SmPr_{/Z}$ denote the category of smooth proper schemes.  The homological motives functor lifts naturally to a functor 
\[
C_*  = C_*^{\MChC}(-,\QQ):
\SmPr_{/Z} \to \MChC(Z,\QQ).
\]
If $X$ is smooth and proper over $Z$, the counit 
\[
C_*(X) \to C_*(Z) = \QQ(0)
\]
is induced by the structure morphism $X \to Z$, and the comultiplication 
\[
C_*(X) \to C_*(X\times X) =C_*(X) \otimes C_*(X) 
\]
is induced by the diagonal of $X$ and by the K\"unneth isomorphism.

\begin{sprop}[Artin motivic Chow coalgebras]
\label{mcffe}
Let $Z = \Spec k$, $k$ a field. If $X,Y$ are finite \'etale over $k$ then
\[
\Hom_{\MChC(Z,\QQ)}(C_*X, C_*Y) = \Hom_Z(X,Y).
\] 
\end{sprop}

The proof, which amounts to a simple computation, spans segments \ref{mosc2}-\ref{mosc3}.

\segment{mosc2}{}
For any $k$-scheme $T$ we let $Z_i(T)$ denote the $\QQ$-vector space of dimension $i$ cycles on $T$ with $\QQ$-coefficients. We let $\ep_T$ denote the structure morphism of $T$ 
\[
T \to Z,
\]
and  $\delta_T$  the diagonal
\[
T \to T \times_k T,
\]
In terms of the morphisms $\ep$ and $\delta$ we define 
\[
Z_0^\mu (X,Y)
\]
to be the $\QQ$-vector subspace of $Z_0(X \times Y)$ of 0-cycles $C$ such that the two diagrams in $\MChC(Z,\QQ)$ 
\[
\xymatrix{
(C_*X)^{\otimes 2}
	\ar[r]^{C^{\otimes 2}} &
C_*Y^{\otimes 2} &
C_*X \ar[rr]^C \ar[dr]_{\ep_X} &
&
C_*Y \ar[dl]^{\ep_Y}
\\
C_*X \ar[r]_C \ar[u]^{\delta_X} &
C_*Y \ar[u]_{\delta_Y} &
&
\QQ(0)
}
\]
commute.

Since $W := X \times_k Y$ is zero dimensional, we have
\[
Z_0(W) = \Ch_0(W)
\]
 inducing an equality of subspaces 
 \[
 Z_0^\mu(X,Y) = \Hom_{\MChC}(C_*X, C_*Y).
 \]
 
 \begin{slm}
 \label{mosc2.1}
 Fix an algebraic closure $\kbar$ of $k$ and let $G$ denote the corresponding Galois group. Then
 \[
 \Hom_k(X,Y) = \Hom_\kbar(X_\kbar, Y_\kbar)^G.
 \]
 \end{slm}
 
 \begin{proof}
 This may be viewed as a direct application of Grothendieck's main theorem of Galois theory. It also follows from fpqc descent by a familiar computation. 
 \end{proof}
 
 \begin{slm}
 \label{mosc2.2}
 In the situation and the notation of segments \ref{mcffe}--\ref{mosc2.1}, we have
 \[
 Z_0(W) = Z_0(W_\kbar)^G.
 \]
 \end{slm}
 
 \begin{proof}
 For any scheme $T$ we let $\cl (T)$ denote the set of closed points of $T$. Then
 \[
 Z_0(W) = \QQ\langle \cl W \rangle
 \]
 and 
 \[
 Z_0(W_\kbar) = \QQ\langle \cl W_\kbar \rangle.
 \]
 So it's enough to note that 
 \[
 \cl W = (\cl W_\kbar)^G.
 \]
 Indeed, if $k'$ is a finite field extension of $k$, then
 \[
 k' \otimes \kbar = \kbar^{\Hom_k(k', \kbar)}
 \]
 and $G$ acts transitively on the latter.
 \end{proof}
 
 \begin{slm}
 \label{mosc2.3}
Continuing with the situation and the notation of segments \ref{mcffe}--\ref{mosc2.1}, the isomorphism of lemma \ref{mosc2.2} restricts to a bijection
 \[
 Z_0^\mu(X,Y) = Z_0^\mu(X_\kbar, Y_\kbar)^G.
 \]
 \end{slm}
 
 \begin{proof}
 Compatibility of a 0-cycle in $Z_0(W)$ with diagonals is an equation inside $Z_0( X \times Y^2)$. Since
 \[
 Z_0 (X\times Y^2) \to
  Z_0 \big( (X\times Y^2)_\kbar \big)
 \]
 is injective, the compatibility may be checked after base-change to $\kbar$.
 
 Similarly, compatibility with counits is an equation in $Z_0(X)$ so may be checked after base-change to $\kbar$.
 \end{proof}
 
 \segment{mosc2.4}{}
 We continue the proof of theorem \ref{mcffe}. By lemmas \ref{mosc2.1}--\ref{mosc2.3}, it's enough to establish a $G$-equivariant bijection 
 \[
 \Hom_\kbar(X_\kbar, Y_\kbar) = 
 	Z_0^\mu(X_\kbar, Y_\kbar).
 \]
For the $G$-equivariance, note simply that if $Z' \to Z$ is a $Z$-scheme and $X', Y'$ denote the base-change of $X,Y$ to $Z'$, then the map
\[
\Hom_{Z'}(X',Y') \to Z_0(X'\times Y')
\]
is natural in $Z'$.

\segment{mosc2.5}{}
Finally, we assume $k = \overline k$ and set out to establish a bijection
\[
\Hom_k(X, Y) = 
 	Z_0^\mu(X, Y).
\]
For this, we note that the category of Chow motives of finite \'etale $k$-schemes is naturally isomorphic to the category $\QQ\langle \mbox{fin. set} \rangle$ whose objects are finite sets and whose morphisms are matrices with entries in $\QQ$. Explicitly, if $S$ and $T$ are finite sets, then
\[
\Hom(S,T) = \Hom_{\Set} (T \times S, \QQ).
\]
We replace $C_*X$ with its image in $\QQ\langle \mbox{fin. set} \rangle$, namely, the set of connected components of $X$, and similarly for $C_*Y$.

Fix a morphism
\[
C = ( c_{y,x} )_{x\in X, y\in Y}
\]
from $C_*X \to C_*Y$.

\segment{mosc2.6}{}
We consider first the compatibility with counits. The graph of the structure morphism $X \to Z$ is $X$ itself, so gives rise to the cycle $\sum_{x\in X}x$ where the sum runs over the set of closed points of $X$. Hence, the counit of $C_*X$ in $\QQ\langle \mbox{fin. set} \rangle$ is given by the row-vector $(1, \dots, 1)$, and the associated compatibility condition on $C$ is given by
\[
\begin{pmatrix}
1 & \cdots & 1
\end{pmatrix}
=
\begin{pmatrix}
1 & \cdots & 1
\end{pmatrix}
\begin{pmatrix}
c_{y,x}
\end{pmatrix}_{x \in X, y\in Y}.
\]
Hence
\[
\tag{$\ep$}
\sum_{y \in Y} c_{y,x} = 1.
\]

\segment{mosc2.7}{}
We turn to compatibility with comultiplication. We begin by transporting the tensor structure of Chow motives to $\QQ\langle \mbox{fin. set} \rangle$: the tensor product of objects is given by Cartesian product of sets, and tensor product of morphisms is the tensor product of matrices. Explicitly, given morphisms
\begin{align*}
a: S\to S', && b: T \to T',
\end{align*}
the morphism
\[
a \otimes b : S\times T \to S' \otimes T\
\]
is given by the formula
\[
(a \otimes b)_{(s',t'),(s,t)} := a_{s',s} b_{t',t'}.
\]
In particular, 
\[
C^2: C_*X^2 \to C_*Y
\]
is given by
\[
(c^{\otimes 2})_{(y,y'),(x,x')} = c_{y,x} \cdot c_{y',x'}.
\]

Meanwhile, the comultiplication morphisms are given by
\begin{align*}
\delta_X: C_*X \to C_*X^2 
&&
\delta_{(x',x''),x}
=
\left\{
\begin{matrix}
1 & \mbox{if} & x'' = x' = x
\\
0 &\mbox{} & \mbox{otherwise,}
\end{matrix}
\right.
\end{align*}
and similarly for $Y$. Hence, the compatibility amounts to
\begin{align*}
c_{y,x}\cdot c_{y',x} &=
	c^{\otimes 2}_{(y,y'),(x,x)}
\\
	&= \sum_{(x',x'') \in X^2} 
		c^{\otimes 2}_{(y,y'),(x',x'')} 
		\cdot
		\delta_{(x',x''),x}
\\
	&= (c^{\otimes 2} \cdot \delta_X)_{(y',y''),x} 
\\
	&= ( \delta_Y \cdot c ) _{(y,y'),x}
\\
	&= \sum_{y'' \in Y}
		\delta_{(y,y'),y''} \cdot c_{y'',x}
\\
	&= 
	\left\{
	\begin{matrix}
	c_{y,x} & \mbox{if} & y=y'
	\\
	0 &\mbox{if} & y \neq y'.
	\end{matrix}
	\right.
\end{align*}
Thus, separating the two cases, we obtain two families of equations 
\[
\tag{$\delta_1$}
(c_{y,x})^2 = c_{y,x}
\]
and
\[
\tag{$\delta_2$}
c_{y,x} \cdot c_{y', x} = 0 
\quad
\mbox{if}
\quad
y \neq y'.
\]

\segment{mosc3}{}
Together, equations ($\epsilon$), ($\delta_1$), and ($\delta_2$) imply that for each closed point $x\in X$ there's a unique closed point $y\in Y$ such that
\[
c_{y',x} = 
\left\{
\begin{matrix}
1 & \mbox{if} & y'=y
\\
0 &\mbox{if} & y' \neq y.
\end{matrix}
\right.
\]
This means precisely that the matrix $C$ corresponds to the graph of a morphism $X \to Y$, and completes the proof of theorem \ref{mcffe}.

\begin{sthm}[Artin motivic dga's]
\label{mdffe}
Let $Z = \Spec k$, $k$ a field. If $X,Y$ are finite \'etale over $k$ then
\[
\Hom_{\DMdga(Z,\QQ)}(C^*Y, C^*X) = \Hom_Z(X,Y).
\]
\end{sthm}

\begin{proof}
We have
\begin{align*}
\tag{$*$}
\Hom_{\DMdga(Z,\QQ)}(C^*Y, C^*X) &= 
	\pi_0\Hom_{\DDMdga(Z,\QQ)}(C^*Y, C^*X)
\\
	&= \pi_0 
 \lim_{\Delta} \lim_{S = (S_1 \to \cdots \to S_k)}
 \Hom_{\DDA}(C^*X^{\otimes S_1}, C^*Y).
\end{align*}
By the smooth proper case of theorem \ref{mosc1}, 
\[
\Hom_{\DDA(Z,\QQ)}
(C^*X^{\otimes S_1}, C^*Y)
=
\Kk^{(0)}(X^{S_1} \times Y),
\]
hence 
\[
\pi_i  \Hom_{\DDA}(C^*X^{\otimes S_1}, C^*Y) = 
K_i^{(0)}(X^{S_1} \times Y) 
\]
vanishes for $i \neq 0$. Indeed, this holds for any scheme and appears to be an in-built feature of the original K-theoretic construction; see \cite[\S3]{SchneiderIntroBeilinson} and the references \cite{Kratzer, Hiller} given there.\footnote{There is also a purely motivic explanation. Motivic cohomology is also the cohomology of certain complexes of Zariski sheaves; $\QQ(0)$ corresponds to the constant sheaf $\QQ$ concentrated in degree $0$.} Consequently, 
\[
(*) = \lim_{k \in \Delta} 
\prod_{[S = (S_1 \to \cdots \to S_k)]}
 \Hom_{\DA}(C^*X^{\otimes S_1}, C^*Y) ^{\Aut S}
 \]
where the solid ``D'' indicates that we're in the homotopy category. Being a limit of \textit{sets}, the limit over $\Delta$ is simply an equalizer, which by direct computation is equal to

\begin{align*}
& \Hom_{\CAlg \DA} (C^*Y, C^*X)
 \\
 &= \Hom_{\MChC}(C_*X, C_*Y)
 \\
 &= \Hom_Z(X,Y)
\end{align*}
as claimed.
\end{proof}
\bibliography{NMH_refs}
\bibliographystyle{plain}
\vfill
\Small\textsc{
ID: Department of Mathematics\\Ben-Gurion University of the Negev
} 
\\
\texttt{ishaidc@gmail.com}
\\

\smallskip

\noindent
\Small\textsc{
TS: Einstein Institute of Mathematics
\\
Hebrew University of Jerusalem
}
\\
\texttt{tomer.schlank@mail.huji.ac.il}

\end{document}